\documentclass[12pt]{amsart}
\usepackage{amssymb,amsmath,amsthm, enumerate}
\newtheorem{theorem}{Theorem}
\newtheorem*{theorem nonum}{Theorem}

\newtheorem{proposition}[theorem]{Proposition}
\newtheorem{corollary}[theorem]{Corollary}
\newtheorem{definition}[theorem]{Definition}
\theoremstyle{remark}
\newtheorem*{example}{Example}
\newtheorem*{remark}{Remark}

\numberwithin{theorem}{section} \numberwithin{equation}{section}
\newcommand{\artin}[2]{\genfrac{(}{)}{}{}{#1}{#2}}
\providecommand{\abs}[1]{\lvert#1\rvert}

\begin{document}
\title{Smallest Irreducible of the Form $x^2-dy^2$}
\author{Shanshan Ding}
\address{2994 Lerner Hall, Columbia University, New York, NY 10027} \email{sd2204@columbia.edu}

\begin{abstract}
It is a classical result that prime numbers of the form $x^2+ny^2$ can be characterized via class field theory for an infinite set of $n$.  In this paper we derive the function field analogue of the classical result.  Then we apply an effective version of the Chebotarev density theorem to bound the degree of the smallest irreducible of the form $x^2-dy^2$, where $x$, $y$, and $d$ are elements of a polynomial ring over a finite field.
\end{abstract}
\maketitle

\section{Introduction and Statement of Results}
Mathematicians since Fermat have studied primes of the form $x^2+ny^2$.  This is a seemingly simple topic that quickly taps into some of the deepest subjects in number theory.  In his book fittingly titled \textit{Primes of the Form} $x^2+ny^2$, David Cox \cite{C} applied extensive concepts from class field theory and complex multiplication to address this topic in the number field setting.  A particularly important result is the following:

\begin{theorem nonum}[Cox 5.26]
Suppose $n > 0$ is a square-free integer such that $n \not\equiv 3 \pmod 4$.  If $p$ is an odd prime not dividing $n$, then $p$ is of the form $x^2+ny^2$ for some $x, y \in \mathbb{Z}$ if and only if the ideal in $\mathbb{Z}$ generated by $p$ splits completely in the Hilbert class field of $\mathbb{Q}(\sqrt{-n})$.
\end{theorem nonum}

This is a special case of a well-known general result in class field theory for Dedekind domains.  In Theorem \ref{SplCon}, we adapt it to function fields to characterize irreducibles that generate ideals of the form $(x^2-dy^2)$, where $x$, $y$, and $d$ are elements of a polynomial ring over a finite field.  The objective of this paper is to bound the degree of the smallest irreducible of the form $x^2-dy^2$ in terms of $\deg d$ and the size of the constant field.  We will accomplish this by applying an effective version of the Chebotarev density theorem to Theorem \ref{SplCon}.  

Throughout this paper, $\mathbb{F}_q$ will denote a finite field of $q$ elements, where $q$ is a power of an odd prime. Let $A=\mathbb{F}_q[T]$ be the polynomial ring over $\mathbb{F}_q$, $F=\mathbb{F}_q(T)$ be the quotient field of $A$, $K=F(\sqrt{d})$ for some $d\in A$ be a quadratic extension of $F$, and $L$ be the Hilbert class field of $K$. We have $F\subset K\subset L$ as a tower of global function fields.  Finally, let $B_K$ and $B_L$ denote the integral closures of $A$ in $K$ and $L$, respectively, and observe that $B_K=A[\sqrt{d}]$ if $d$ is square-free.  

We now state our main result.  Unless otherwise specified, we will work within the setting outlined in the previous paragraph.

\begin{theorem}\label{bound}
Suppose $d\notin \mathbb{F}_q$ is square-free, and suppose further that if $\deg d$ is even, then the leading coefficient of $d$ is a square in $\mathbb{F}^{\times}_q$.  Under these conditions, there is an irreducible $p\in A$ of the form $x^2-dy^2$ for some $x$, $y\in A$ such that 
\begin{equation*}
\deg p \le \left\lceil \frac{4\log (2\hat{r}\left\lceil\frac{\deg d}{2}\right\rceil+2)}{\log q} \right\rceil, 
\end{equation*}
\[
  \ \text{where } \hat{r}=\begin{cases}
   	(\sqrt{q}+1)^{\deg d-1} &\text{ if } \deg d \text{ is odd}\\
   	\frac{2(\sqrt{q}+1)^{\deg d-2}}{\deg d} &\text{ if } \deg d \text{ is even}.
    \end{cases}
\]
\end{theorem}
   
Let $h_K$ denote the divisor class number of $K$.  We can derive from Theorem \ref{bound} a non-trivial lower bound for $h_K$. 
\begin{corollary}\label{app}
If $\deg d$ is odd, then
\begin{equation*}
h_K > \frac{q^{\frac{\deg d-1}{4}}-2}{\deg d+1}.
\end{equation*}
\end{corollary}

\section{Preliminaries}
Every finite prime of $F$ is of the form $A_{(p)}$, the discrete valuation ring (DVR) obtained by localizing $A$ at some non-zero prime ideal $(p)$ of $A$.  The \emph{infinite prime} of $F$, $p_\infty$, is the localization of the ring $\mathbb{F}_q[T^{-1}]$ at the prime ideal generated by $T^{-1}$.  Its degree is defined to be 1.  For a quadratic extension $K=F(\sqrt{d})$ of $F$, the ramification behavior of $p_\infty$ in $K$ can be easily determined (see 14.6 of \cite{RB}).  If $\deg d$ is odd, then $p_\infty$ ramifies in $K$.  If $\deg d$ is even, then $p_\infty$ splits completely in $K$ if the leading coefficient of $d$ is a square in $\mathbb{F}^{\times}_q$ and remains a prime in $K$ otherwise.  Let $S_\infty$ be the set of primes in $K$ that lie above $p_\infty$. 

\begin{definition}
The Hilbert class field $L$ of $K$, with respect to $B_K$, is the maximal unramified abelian Galois extension of $K$ in the separable closure of $K$ in which every element of $S_\infty$ splits completely. 
\end{definition} 
\begin{remark}
The extension $L/F$ is Galois (see 2.3 of \cite{RP}).
\end{remark}  

Let $E_1$ and $E_2$ be function fields.  The extension $E_2/E_1$ is \textit{geometric} if $E_1$ and $E_2$ have the same constant field.  It is easy to see that in order for $K/F$ to be a geometric extension, $d$ cannot be a non-square constant.  Assuming that $\mathbb{F}_q$ is indeed the constant field of $K$, then the constant field of $L$ is $\delta$-dimensional over $\mathbb{F}_q$, where $\delta$ is the gcd of the degrees of elements in $S_\infty$ (see 1.3 of \cite{RP}).  Because $\delta=1$ if $p_\infty$ either ramifies or splits completely in $K$ and $\delta=2$ if $p_\infty$ remains a prime in $K$, 
\begin{equation}
L/K \text{ is geometric} \iff p_\infty \text{ ramifies or splits completely in } K.
\label{geo}
\end{equation}
Notice the conditions of Theorem \ref{bound} ensure that $L/F$ is geometric.  

If $\mathfrak{p}$ is a prime ideal of $B_K$ that is unramified in $L$, then there exists a unique $\sigma\in \text{Gal}(L/K)$ such that for all $\alpha\in B_L$, 
\begin{equation}
\sigma(\alpha) \equiv \alpha^{|B_K/\mathfrak{p}|} \!\!\!\pmod{\mathfrak{P}},
\end{equation}
where $\mathfrak{P}$ is a prime ideal in $B_L$ that lies above $\mathfrak{p}$.
The unique $\sigma\in \text{Gal}(L/K)$ is called the \textit{Artin symbol} and denoted by $\artin{L/K}{\mathfrak{P}}$.  Because the Artin symbols of all prime ideals $\mathfrak{P}$ that lie above $\mathfrak{p}$ form a conjugacy class in Gal$(L/K)$, we denote the Artin symbol by $\artin{L/K}{\mathfrak{p}}$ instead to emphasize the underlying prime.  It is a well-known fact that the order of each element in the conjugacy class $\artin{L/K}{\mathfrak{p}}$ is equal to the degree of the extension $B_K/\mathfrak{p} \subset B_L/\mathfrak{P}$ (the proof in the number field case can be found in 5.21 of \cite{C}, and the proof in the function field case is completely analogous).  Consequently, 
\begin{equation}
\mathfrak{p} \textit{\emph{ splits completely in} L} \iff \artin{L/K}{\mathfrak{p}}=1.
\label{ArCon}
\end{equation}
We note here that the Artin symbol is well-defined for any finite Galois extension.  Everything in this paragraph would still hold if we replaced $K$ by $F$.   

Next, let $I_K$ be the set of fractional ideals of $B_K$ and $P_K$ be the set of principal fractional ideals in $I_K$.  The quotient group $I_K/P_K$ is the \textit{ideal class group} of $K$.  It is a standard result that $I_K/P_K$ is finite.  We now state another famous result about the ideal class group, the proof of which can be found in 1.3 of \cite{RP}.  Its corollary follows immediately. 
\begin{theorem}\label{ArIso}
The Artin symbol induces an isomorphism between the ideal class group of $K$ and \emph{Gal}$(L/K)$.
\end{theorem}
\begin{corollary}\label{ArSC}
If $\mathfrak{p}$ is a prime ideal in $B_K$, then 
\begin{center}
$\mathfrak{p} \text{ splits completely in L} \iff \mathfrak{p} \text{ is principal}$.
\end{center}
\end{corollary}

Observe that since $[L:K]= \abs{\text{Gal}(L/K)}=\abs{I_K/P_K}$, we know $[L:K]$ is finite.  We call $\abs{I_K/P_K}$ the \textit{ideal class number} of $K$ and denote it by $h_{B_K}$, which is not to be confused with the \textit{divisor class number} of $K$, denoted by $h_K$.  The relationship between $h_{B_K}$ and $h_{K}$ (see 14.7 of \cite{RB}) depends on the ramification behavior of $p_\infty$ in $K$ and can be summarized as the following:
\begin{equation}
  \ h_{B_K}=\begin{cases}
    \ h_{K} &\text{ if $p_\infty$ ramifies}\\
   	\ \frac{h_{K}}{\deg g} &\text{ if $p_\infty$ splits completely}\\
   	\ 2h_{K} &\text{ if $p_\infty$ remains a prime},
    \end{cases}
\label{class number}
\end{equation}
where in the second case $g\in A$ and $g+h\sqrt{d}$ for some $h\in A$ is a fundamental unit in $B_K$.  We will use (\ref{class number}) to bound the ideal class number in Section \ref{proof}.

A detailed proof of Cox 5.26 is presented in \cite{C}; here we outline the proof to motivate the function field analogue of the theorem.  Let $\mathcal{O}_K$ denote the ring of algebraic integers in $K=\mathbb{Q}(\sqrt{-n})$.  If $n$ is a positive, square-free integer and $n \not\equiv 3 \pmod 4$, then $\mathcal{O}_K = \mathbb{Z}[\sqrt{-n}]$.  Furthermore, if $p \nmid n$, then $(p)$ is unramified in $K$.  Because $(x^2+ny^2)=(x+y\sqrt{-n})(x-y\sqrt{-n})$ in $\mathcal{O}_K$, 
\begin{equation}
\begin{split}
(p)=(x^2+ny^2)&\iff p\mathcal{O}_K=\mathfrak{p} \bar{\mathfrak{p}}, \mathfrak{p}\not= \bar{\mathfrak{p}}, \text{and } \mathfrak{p} \text{ is principal in }\mathcal{O}_K\\ 
&\iff p\mathcal{O}_K=\mathfrak{p} \bar{\mathfrak{p}}, \mathfrak{p}\not= \bar{\mathfrak{p}}, \text{and } \mathfrak{p} \text{ splits completely in }L\\
&\iff (p) \text{ splits completely in }L,
\end{split}
\label{keyf}
\end{equation}
where $L$ is the Hilbert class field of $\mathbb{Q}(\sqrt{-n})$.  Recall that for a quadratic function field extension $K=F(\sqrt{d})$ of $F$, if $d$ is square-free, then the integral closure of $A$ in $K$ is $A[\sqrt{d}]$.  Note also that if $B_K=A[\sqrt{d}]$ and $p \nmid d$, then the equivalences in (\ref{keyf}) hold in the function field setting.  Thus we state the following theorem without proof, for its proof is completely analogous to that of Cox 5.26.

\begin{theorem}[Analogue of Cox 5.26]\label{SplCon}
Let $d\in A$ be square-free \emph{(}$d\neq 0, 1$\emph{)}.  If $p\in A$ is an irreducible not dividing $d$, then 
\begin{center}
$(p)=(x^2-dy^2) \text{ for some } x, y\in A \iff (p) \text{ splits completely in } L$. 
\end{center}
\end{theorem}

\section{Chebotarev Density Theorem and Proof of Theorem \ref{bound}}\label{proof}
Let $p$ be a prime ideal in a number field $E_1$.  The classical Chebotarev density theorem states that for a Galois extension $E_2/E_1$, if $C$ is a conjugacy class in $G=\text{Gal}(E_2/E_1)$, then the Dirichlet density of the set $\{p\subset E_1 \mid p \text{ unramified in } E_2, \artin{E_2/E_1}{p}=C\} \text{ is } \frac{\abs{C}}{\abs{G}}$.  One could use this result to approximate the number of unramified primes of a given degree whose Artin symbols are in the same conjugacy class.  Effective versions of the Chebotarev density theorem bound the error term of this approximation, which in the number field case was addressed by Lagarias and Odlyzko \cite{LO}.  Murty and Scherk \cite{MS} provided analogues of their results for function fields, where the Riemann hypothesis is known to be true.

Suppose $E_1$ and $E_2$ are function fields.  Let $\mathbb{F}$ be the constant field of $E_1$, $\bar{\mathbb{F}}$ be the algebraic closure of $\mathbb{F}$ in $E_2$, and let $m$ denote $[\bar{\mathbb{F}}:\mathbb{F}]$.  Define 
\begin{equation*}
\begin{split}
&\text{$\pi(n):=\# \{ p\subset E_1\mid p \text{ unramified in $E_2$, } \deg p=n\}$ and} \\
&\text{$\pi_C(n):=\# \{ p\subset E_1\mid p \text{ unramified in $E_2$, }  \deg p=n, \artin{E_2/E_1}{p}=C\}$.}
\end{split}
\end{equation*}
Murty and Scherk showed that
\begin{equation}
\left\lvert \pi_C(n) - m \frac{\abs{C}}{\abs{G}} \pi(n) \right\rvert \le 2g_{E_2} \frac{\abs{C}}{\abs{G}} \frac{q^{n/2}}{n} + 2(2g_{E_1}+1) \abs{C} \frac{q^{n/2}}{n} + \left(1+ \frac{\abs{C}}{n}\right) \abs{D_{E_2/E_1}},
\label{CDT} 
\end{equation}
where $g_{E_i}$ is the genus of $E_i$, $q= \abs{\mathbb{F}}$, and $\abs{D_{E_2/E_1}}= \!\!\!\sum\limits_{\substack{p \subset E_1\\ \text{ram. in } E_2}}\!\!\!\!\deg p$ is the degree of the different of $E_2$ over $E_1$.  Note here that $\pi_C(n)$ is certain to be positive as soon as $m \frac{\abs{C}}{\abs{G}} \pi(n)>$ RHS of (\ref{CDT}).

Now let $E_1=F$, $E_2=L$, and $C=1$.  By (\ref{ArCon}) and Theorem \ref{SplCon}, if $d$ is square-free and $p$ is an irreducible in $A$ not dividing $d$, then 
\begin{equation}
(p)=(x^2-dy^2) \iff \artin{L/F}{(p)} = C.
\end{equation}
Since DVRs that arise from non-zero prime ideals of $A$ exhaust the finite primes of $F$, $\pi_C(n)$ represents the number of prime ideals of degree $n$ ($n>0$) in $A$ of the form $(x^2-dy^2)$, plus possibly 1 for the infinite prime $p_\infty$ if $n=1$.  Thus the smallest positive integer $n$ such that $m \frac{\abs{C}}{\abs{G}} \pi(n)> \text{RHS of (\ref{CDT})}+1$ is an upper bound for the degree of the smallest $(p)$ in $A$ of the form $(x^2-dy^2)$.  Because we can always multiply $p$ by the unit $\frac{x^2-dy^2}{p}$, this is equivalent to an upper bound for the degree of the smallest irreducible in $A$ of the form $x^2-dy^2$.  

Our goal therefore is to find an upper bound for the first $n$ such that $m \frac{\abs{C}}{\abs{G}} \pi(n)> \text{RHS of (\ref{CDT})}+1$ in terms of $\deg d$ and $q$.  To do this, we must first decipher the terms that appear in (\ref{CDT}).  Clearly, $\abs{C}=1$ and $\abs{G}=2r$, where $r=[L:K]$ will be addressed later.  It is a well-known fact (see p.49 of \cite{RB}) that the genus of $F$ is 0.  To find $g_L$ and $\abs{D_{E_2/E_1}}$ we resort to the Riemann-Hurwitz theorem for function fields (see 7.16 of \cite{RB}), which states that if $E_2/E_1$ is a finite, separable, geometric extension, then
\begin{equation}
2g_{E_2}-2=[E_2:E_1](2g_{E_1}-2)+\abs{D_{E_2/E_1}}.
\label{RH}
\end{equation} 
Because the condition of the Riemann-Hurwitz theorem requires $F\subset K\subset L$ to be geometric extensions, we require that $d\notin \mathbb{F}_q$, and if $\deg d$ is even, we also require the leading coefficient of $d$ to be a square in $\mathbb{F}^{\times}_q$.  

To find $g_L$, we will apply (\ref{RH}) twice, first to $K/F$ to solve for $g_K$, then to $L/K$.  Having solved for $g_L$, we can then apply (\ref{RH}) once more to $L/F$ to find $\abs{D_{L/F}}$.  The following proposition will help us with the first step.

\begin{proposition}
If $d$ satsifies the hypotheses of Theorem \ref{bound}, then
\[
  \ \abs{D_{K/F}}=\deg d + \begin{cases}
		1 &\text{if $\deg d$ is odd}\\
		0 &\text{if $\deg d$ is even}. 
    \end{cases}
\]
\end{proposition}

\begin{proof}
Given that $\abs{D_{K/F}}= \!\!\!\sum\limits_{\substack{p\subset F\\ \text{ram. in } K}}\!\!\!\!\deg p$, let $d=ud_1d_2\cdots d_l$ be the factorization of $d$ into unit and monic irreducibles in $A$, then $(d)=(d_1)(d_2)\cdots (d_l)=(\sqrt d)^2$ as ideals in $F(\sqrt{d})=K$.  Since $d_1$, $d_2$, \ldots, $d_l$ are pairwise coprime, the complete factorization of $(d)$ in $K$ must be $(\sqrt{d_1})^2 (\sqrt{d_2})^2 \cdots (\sqrt{d_l})^2$, so the finite primes of $F$ that ramify in $K$ are precisely the DVRs that arise from $(d_1)$, $(d_2)$, \ldots, $(d_l)$.  Thus $\abs{D_{K/F}}= \sum\limits_{d_i} \deg d_i=\deg d$, plus 1 if $p_\infty$ ramifies in $K$, i.e. if $\deg d$ is odd.     
\end{proof}

Since $\abs{D_{L/K}}=0$ by definition of the Hilbert class field, we compute from the Riemann-Hurwitz theorem that
\begin{equation}
\begin{split}
&g_K=\left\lceil\frac{\deg d}{2}\right\rceil-1 \text{, }\\ &g_L=r\left(\left\lceil\frac{\deg d}{2}\right\rceil-2\right)+1 \text{, and }\\
&\abs{D_{L/F}}=2r\left\lceil\frac{\deg d}{2}\right\rceil.
\end{split}
\end{equation}

Next we deal with the term $m \frac{\abs{C}}{\abs{G}} \pi(n)$.  Since $L/F$ is geometric, $m=1$.  
\begin{proposition}
As it appears in \emph{(}\ref{CDT}\emph{)}, 
\begin{equation*}
\pi(n) \ge \frac{q^n}{n} - \frac{q^{n/2}}{n} - q^{n/3} - \frac{2r\left\lceil\frac{\deg d}{2}\right\rceil}{n}.
\end{equation*}
\end{proposition}

\begin{proof}
Define  $\gamma_n$ to be the number of monic irreducibles in $A$ and $\varepsilon_n$ the number of primes in $F$ that ramify in $L$, both of degree $n$.  Observe that $\pi(n)=\gamma_n - \varepsilon_n$, plus 1 if $p_\infty$ ramifies.  It is well-known (see p.14 of \cite{RB}) that $\abs{\gamma_n - \frac{q^n}{n}} \le \frac{q^{n/2}}{n} + q^{n/3}$, so $\gamma_n \ge \frac{q^n}{n} - \frac{q^{n/2}}{n} - q^{n/3}$.  Since $\abs{D_{L/F}}= \!\!\!\sum\limits_{\substack{p\subset F\\ \text{ram. in } L}}\!\!\!\!\deg p$, $\varepsilon_n \le \frac{\abs{D_{L/F}}} {n} = \frac{2r\left\lceil\frac{\deg d}{2}\right\rceil}{n}$.  
\end{proof}

We are now ready to bound the smallest positive integer $n$ such that $m \frac{\abs{C}}{\abs{G}} \pi(n)> \text{RHS of (\ref{CDT})}+1$ in terms of $\deg d$, $q$, and $r$. 

\begin{proof}[Proof of Theorem \ref{bound}]
After substituting and collecting terms, we need to bound the smallest $n$ such that 
\begin{equation}
\begin{split}
q^n-\left(3+2r\left\lceil\frac{\deg d}{2}\right\rceil\right)q^{n/2} -nq^{n/3} &-\left(4r^2\left\lceil\frac{\deg d}{2}\right\rceil + 2r\right)n \\ &-\left(4r^2\left\lceil\frac{\deg d}{2}\right\rceil + 2r\left\lceil\frac{\deg d}{2}\right\rceil\right) > 0.
\end{split}
\end{equation} 
Note that $q^{n/2}>n$ for all ($q, n$) and $2q^{n/2}\geq nq^{n/3}$ for all ($q, n$)$\neq$($3, 5$).  If ($q, n$)$=$($3, 5$), then $nq^{n/3}-2q^{n/2}<(q^{n/2}-n)\left(4r^2\left\lceil\frac{\deg d}{2}\right\rceil + 2r\right)$.  These observations show that we can reasonably bound $n$ by solving for it in the equation
\begin{equation}
q^n-\left(4r^2\left\lceil\frac{\deg d}{2}\right\rceil + 2r\left\lceil\frac{\deg d}{2}\right\rceil + 2r + 5\right)q^{n/2} - \left(4r^2\left\lceil\frac{\deg d}{2}\right\rceil + 2r\left\lceil\frac{\deg d}{2}\right\rceil\right)=0,
\end{equation}
to which we can apply the quadratic formula and conclude that the smallest $n$ in question satisfies 
\begin{equation}
q^{n/2}< \left(2r\left\lceil\frac{\deg d}{2}\right\rceil+2\right)^2.
\end{equation}
Since $n$ is a positive integer, 
\begin{equation}
n \le \left\lceil \frac{4\log (2r\left\lceil\frac{\deg d}{2}\right\rceil+2)}{\log q} \right\rceil.
\label{main bound}
\end{equation}

All that remains is finding an upper bound for $r$ in terms of $\deg d$ and $q$, which we denote by $\hat{r}$ to emphasize the fact that $[L:K]$ itself may be much smaller.  In particular, if the infinite prime of $F$ splits completely in $K$, it could very well be the case that $[L:K]=1$.   

By (\ref{class number}), if $\deg d$ is odd, then $r=h_K$, where $h_K$ is the divisor class number of $K$.  Since we have excluded the possibility that $p_\infty$ remains a prime in $K$, if $\deg d$ is even, then $r=\frac{h_K}{\deg g}$, where $g+h\sqrt{d}$ is a fundamental unit in $B_K$.  Observe that $g^2-dh^2\in \mathbb{F}^{\times}_q$, so $\deg g\ge \frac{\deg d}{2}$, and consequently $r\le \frac{2h_K}{\deg d}$ if $\deg d$ is even.  A well-known bound on $h_K$ (see 5.11 of \cite{RB}) is $(\sqrt{q}-1)^{2g_K} \le h_K \le (\sqrt{q}+1)^{2g_K}$, hence
\begin{equation}
\ r\le \begin{cases}
		(\sqrt{q}+1)^{\deg d-1} &\text{ if } \deg d \text{ is odd}\\
   	\frac{2(\sqrt{q}+1)^{\deg d-2}}{\deg d} &\text{ if } \deg d \text{ is even},   
   	\end{cases}
\end{equation}
and we are done.
\end{proof}

\begin{remark}
Given large enough $q$, Theorem \ref{bound} suggests we can expect the degree of the smallest irreducible polynomial of the form $x^2-dy^2$ to be bounded by roughly $2\deg d$.
\end{remark}

\begin{proof}[Proof of Corollary \ref{app}]
If $\deg d$ is odd, then the degree of the smallest irreducible of the form $x^2-dy^2$ must be at least $\deg d$, thus we obtain Corollary \ref{app} by rearranging the terms in (\ref{main bound}).
\end{proof}
\begin{remark}
If $q$ is small, Corollary \ref{app} actually gives a better lower bound on $h_K$ than $(\sqrt{q}-1)^{\deg d-1}$ does for large degrees of $d$.  For $q=3$, this happens if $\deg d \ge 11$, and for $q=5$, if $\deg d \ge 17$.
\end{remark}

\begin{example}
Let $q=5$ and $d=T^{19} + 3T^8 + 2$.  The smallest irreducible in $\mathbb{F}_5[T]$ of the form $x^2-dy^2$ is 
\begin{equation}
(T+2)^2-(T^{19} + 3T^8 + 2)=4T^{19} + 2T^8 + T^2 + 4T + 2.
\end{equation}  
Its degree is 19, which is less than our upper bound of 60 by Theorem \ref{bound}.  Furthermore, we computed in Magma that the class number of $\mathbb{F}_5(T, \sqrt{d})$ is 1348408, which is larger than our lower bound of 70 by Corollary \ref{app}.  In comparison, the lower bound for $h_K$ given by $(\sqrt{q}-1)^{\deg d-1}$ is 46.
\end{example} 

\section*{Acknowledgements}
The author would like to thank Jeremy Rouse for his unfailing guidance throughout every stage of this project, Ken Ono for his insight and support, and the referee for numerous helpful comments.  This research was generously funded through the NSF-REU program.

\end{document}